\documentclass{amsart}

\usepackage{amscd}
\usepackage{amsfonts}
\usepackage{amsthm}
\usepackage{geometry}
\usepackage{amsmath}
\usepackage{txfonts}
\usepackage{graphicx}
\usepackage{wrapfig}
\usepackage{ascmac}
\usepackage{bm}
\usepackage{indentfirst}
\usepackage{comment}

\newtheorem{thm}{Theorem}[section]
\newtheorem{rem}[thm]{Remark}
\newtheorem{defi}[thm]{Definition}

\newtheorem{lem}[thm]{Lemma}

\newtheorem{cor}[thm]{Corollary}

\newtheorem{note}{Notation}

\def\NZQ{\Bbb}
\def\NN{{\NZQ N}}

\def\RR{{\NZQ R}}
\def\CC{{\NZQ C}}

\def\PP{{\NZQ P}}

\def\SS{{\NZQ S}}
\def\TT{{\NZQ T}}

\def\ml{\mathcal{C}}
\def\ml1{\mathcal{C}^1}
\def\mlb1{\mathcal{C}_{b}^{1}}

\def\frk{\frak}

\def\Phi{{\frk n}}

\def\rank{{\rm rank}}

\def\A{{\mathcal A}}
\def\B{{\mathcal B}}

\newcommand{\C}[0]{\mathbb{C}}

\newcommand{\bd}[1]{\textbf{#1}}

\newcommand{\codim}[0]{\mbox{codim }}

\title{Pappus's Theorem in Grassmannian $Gr(3, \CC^n)$}
\author{S. Sawada}
\author{S. Settepanella}
 \author{S. Yamagata}
\address{%
Department of Mathematics,
Hokkaido University, Japan.}
\email{b.lemon329@gmail.com}
\email{s.settepanella@math.sci.hokudai.ac.jp}
\email{so.yamagata.math@gmail.com}

\thanks{The second named author was supported by JSPS Kakenhi Grant Number 26610001.}

\begin{document}

\maketitle


\begin{abstract}
In this paper we study intersections of quadrics, components of the hypersurface in Grassmannian $Gr(3, \CC^n)$ introduced in \cite{SoSuSi}. This lead to an alternative statement and proof of Pappus's Theorem retrieving Pappus's and Hesse configurations of lines as special points in complex projective Grassmannian. This new connection is obtained through a third purely combinatorial object, the intersection lattice of Discriminantal arrangement.
\end{abstract}


\section{Introduction}
Pappus's hexagon Theorem, proved by Pappus of Alexandria in the fourth century A.D., began a long development in algebraic geometry. \begin{center}
\textit{In its changing expressions one can see reflected the changing concerns of the field, from synthetic geometry to projective plane curves to Riemann surfaces to the modern development of schemes and duality (D. Eisenbud, M. Green and J. Harris \cite{EGH}).}
\end{center}
There are several knowns proofs of Pappus's Theorem including its generalizations such as Cayley Bacharach Theorem ( see Chapter $1$ of \cite{Geb} for a collection of proofs of Pappus's Theorem and \cite{EGH} for proofs and conjectures in higher dimension).\\
In this paper, by mean of recent results in \cite{sette} and \cite{SoSuSi}, we connect Pappus's hexagon configuration  
to intersections of well defined quadrics in Grassmannian providing a new statement and proof of Pappus's Theorem as an original result on dependency conditions for defining polynomials of those quadrics. This result enlightens a new connection between special configurations of points ( lines ) in projective plane and hypersurfaces in projective Grassmannain $Gr(3, \CC^n)$.
This connection is made through a third combinatorial object, the intersection lattice of the {\it  Discriminantal arrangement}. Introduced by Manin and Schechtman in 1989, it is an arrangement of hyperplanes generalizing classical braid arrangement (cf. \cite{man} p.209). Fixed a generic arrangement $\A = \{ H^0_1,...,H^0_n\}$ in $\CC^k$,
Discriminantal arrangement $\B(n, k, \A), n, k \in \NN$ for $k \ge 2$ ( $k=1$ corresponds to Braid arrangement ), consists of 
parallel translates $H_1^{t_1},...,H_n^{t_n}, (t_1,...,t_n) \in \CC^n$, of $\A$ which fail to form a generic arrangement in $\CC^k$. The combinatorics of $\B(n, k, \A)$ is known in the case of \textit{very generic arrangements}, i.e. $\A$ belongs to an open Zariski set $\mathcal{Z}$ in the space of generic arrangements $H^0_i,i=1,...,n$  (see \cite{man}, \cite{athana} and \cite{BB}), but still almost unknown for $\A \not\in \mathcal{Z}$.
In 2016, Libgober and Settepanella (cf.\cite{sette}) gave a sufficient geometric condition for an arrangement $\A$ not to be very generic, i.e. $\A \not\in \mathcal{Z}$.
In particular in case $k = 3$, their result shows that multiplicity $3$ codimension $2$ intersections of hyperplanes in $\B(n, 3, \A)$ appears if and only if collinearity conditions for points at infinity of lines, intersections of certain planes in $\A$, are satisfied  ( Theorem 3.8 in \cite{sette}) .
More recently (see \cite{SoSuSi}) authors applied this result to show that points in specific degree 2 hypersurface in Grassmannian $Gr(3, \CC^n)$ correspond to generic arrangements of $n$ hyperplanes in $\C^3$ with associated discriminantal arrangement having intersections of multiplicity 3 in codimension 2 (Theorem 5.4 in \cite{SoSuSi}). 
In this paper we look at Pappus's configuration (see Figure \ref{fig1} ) as a generic arrangement of 6 lines in $\PP^2$ which intersection points satisfy certain collinearity conditions (see Figure \ref{fig2}). This allows us to apply results on \cite{sette} and \cite{SoSuSi} to restate and re-prove Pappus's Theorem.\\
More in details, let $\A $ be a generic arrangement in $\CC^3$ and $\A_{\infty}$ the arrangement of lines in $H_{\infty} \simeq \PP^2$ directions at infinity of planes in $\A$. The space of generic arrangements of $n$ lines in $(\PP^2)^n$ is Zariski open set $U$ in the space of all arrangements of $n$ lines in $(\PP^2)^n$. On the other hand in $Gr(3, \CC^n)$ there is open set $U'$ consisting of $3$-spaces intersecting each coordinate hyperplane transversally (i.e. having dimension of intersection equal $2$). One has also one set $\tilde U$ in $Hom(\CC^3, \CC^n)$ consisting of embeddings with image transversal to coordinate hyperplanes and $\tilde U/GL(3)=U'$ and $\tilde U/(\CC^*)^n=U$. 
Hence generic arrangements in $\CC^3$ can be regarded as points in $Gr(3, \CC^n)$. Let $\{s_1<\dots < s_6\} \subset \{1, \dots, n\}$ be a set of indices of a generic arrangement $\A= \{ H_1^0, \dots, H_n^0\}$ in $\CC^3$, $\alpha_i$ the normal vectors of $H_i^0$'s and $\beta_{ijl} = det (\alpha_i, \alpha_j, \alpha_l)$. For any permutation 
$\sigma \in \mathbf{S_6}$ denote by $[\sigma] =\{\{i_1, i_2\},\{i_3, i_4\},\{i_5, i_6\}\}$, $i_j=s_{\sigma(j)}$, and by  ${\rm Q}_\sigma$ the quadric in $Gr(3, \CC^n)$ of equation $\beta_{i_1 i_3 i_4} \beta_{i_2 i_5 i_6} - \beta_{i_2 i_3 i_4} \beta_{i_1 i_5 i_6} = 0$. The following theorem, equivalent to the Pappus's hexagon Theorem, holds.

\smallskip
\noindent {\it {\bf Theorem 5.3.} (Pappus's Theorem)
For any disjoint classes $[\sigma_1]$ and $[\sigma_2]$, there exists a unique class $[\sigma_3]$ disjoint from $[\sigma_1]$ and $[\sigma_2]$ such that 
\begin{equation*}
{\rm Q}_{\sigma_{i_1}} \cap {\rm Q}_{\sigma_{i_2}} = \bigcap_{i=1}^3 {\rm Q}_{\sigma_i} \quad .
\end{equation*}
for any $\{i_1, i_2\} \subset [3]$.}

\smallskip

\noindent
In the rest of the paper, we retrieve the Hesse configuration of lines studying intersections of six quadrics of the form ${\rm Q}_\sigma$ for opportunely chosen $[\sigma]$. This lead to a better understanding of differences in combinatorics of Discriminantal arrangement in complex and real case. Indeed it turns out that this difference is connected with existence of the Hesse arrangement (see \cite{OT}) in $\PP^2(\CC)$, but not in $\PP^2(\RR)$. \\ 
From above results it seems very likely that a deeper understanding of combinatorics of Discriminantal arrangements arising from non very generic arrangements of hyperplanes in $\CC^k$ ( i.e. $\A \notin \mathcal{Z}$ ), could lead to new connections between higher dimensional special configurations of hyperplanes ( points ) in projective space and Grassmannian. Viceversa, known results in algebraic geometry could help in understanding combinatorics of Discriminantal arrangements in non very generic case. Moreover we conjecture that regularity in the geometry of Discriminantal arrangement could lead to results on hyperplanes arrangements with high multiplicity intersections, e.g. , in case $k=3$, line arrangements in $\PP^2$ with high number of triple points (see Remark 6.6). This will be object of further studies.\\
The content of the paper is the following. \\
In section \ref{pre} we recall definition of Discriminantal arrangement from \cite{man}, basic notions on Grassmannian, and definitions and results from \cite{SoSuSi}.
In section \ref{motivation} we provide an example of the case of 6 hyperplanes in $\CC^3$. 
In section \ref{pappusvariety} we define and study Pappus hypersurface.
Section \ref{pappusthm} contains Pappus's theorem in $Gr(3, \CC^n)$ and its proof.
In the last section we study intersections of higher numbers of quadrics and Hesse configuration.


\section{Preliminaries}\label{pre}

\subsection{Discriminantal arrangement}\label{discarr}

Let $H^0_i, i=1,...,n$ be a generic arrangement in $\CC^k, k<n$ i.e. 
a collection of hyperplanes such that $\codim \bigcap_{i \in K,
 \mid K\mid=p} H_i^0=p$. 
Space of parallel translates $\SS(H_1^0,...,H_n^0)$ (or simply $\SS$ when 
dependence on $H_i^0$ is clear or not essential)
is the space of $n$-tuples
$H_1,...,H_n$ such that either $H_i \cap H_i^0=\emptyset$ or 
$H_i=H_i^0$ for any $i=1,...,n$.
One can identify $\SS$ with $n$-dimensional affine space $\CC^n$ in
such a way that $(H^0_1,...,H^0_n)$ corresponds to the origin.  In particular, an ordering of hyperplanes in $\A$ determines the coordinate system in $\SS$ (see \cite{sette}).\\
We will use the compactification of $\CC^k$ viewing it 
as $\PP^k(\CC)\setminus H_{\infty}$ endowed with collection of hyperplanes
$\bar H^0_i$ which are projecive closures of affine hyperplanes
$H^0_i$. Condition of genericity is equivalent to $\bigcup_i \bar H^0_i$ 
being a normal crossing divisor in $\PP^k(\CC)$.\\
Given a generic arrangement $\A$ in $\CC^k$ 
formed by hyperplanes $H_i, i=1,...,n$
{\it the trace at infinity}, denoted by $\A_{\infty}$,  is the arrangement 
formed by hyperplanes 
$H_{\infty,i}=\bar H^0_i\cap H_{\infty}$.  The trace $\mathcal{A}_{\infty}$ of an arrangement $\mathcal{A}$ determines the space of parallel translates $\mathbb{S}$ (as a subspace in the space of $n$-tuples of hyperplanes in $\mathbb{P}^k$).\\
Fixed a generic arrangement  $\mathcal{A}$, consider the closed subset of $\mathbb{S}$ formed by those collections which fail to form a generic arrangement. This subset is a union of hyperplanes with each hyperplane $D_L$ corresponding to a subset $L = \{ i_1, \dots, i_{k+1} \} \subset$  [$n$] $\coloneqq \{ 1, \dots, n \}$ and consisting of $n$-tuples of translates of hyperplanes $H_1^0, \dots, H_n^0$ in which translates of $H_{i_1}^0, \dots, H_{i_{k+1}}^0$ fail to form a generic arrangement. The arrangement $\B(n, k, \A)$ of hyperplanes $D_L$ is called $Discriminantal$ $arrangement$ and has been introduced by Manin and Schechtman in \cite{man}.  Notice that $\B(n, k, \A)$ only depends on the trace at infinity $\mathcal{A}_{\infty}$ hence it is sometimes more properly denoted by $\B(n, k,\mathcal{A}_{\infty})$.

\subsection{Good 3s-partitions}
Given $s \geq 2$ and $n \geq 3s$, \textit{a good $3s$-partition} ( see \cite{SoSuSi}) is a set $\mathbb{T} = \{ L_1, L_2, L_3 \}$, with $L_i$ subsets of $[n]$ such that $|L_i| = 2s$, $|L_i \cap L_j| = s$ ($i \neq j$), $L_1 \cap L_2 \cap L_3 = \emptyset$ (in particular $|\bigcup L_i| = 3s$), i.e. $L_1 = \{ i_1, \dots, i_{2s} \}, L_2 = \{ i_1,\dots, i_s, i_{2s+1}, \dots, i_{3s} \}, L_3 = \{ i_{s+1}, \dots, i_{3s} \}$. \\
Notice that given a generic arrangement $\mathcal{A} $ in $\mathbb{C}^{2s-1}$, subsets $L_i$ define hyperplanes $D_{L_i}$ in the Discriminantal arrangement $\B(n, 2s-1,\mathcal{A}_{\infty})$. In this paper we are mainly interested in the case $s=2$ corresponding to generic arrangements in $\CC^3$.

\subsection{Matrices $A(\mathcal{A}_\infty)$ and $A_{\mathbb{T}}(\mathcal{A}_\infty)$ }\label{plucker}
Let $\alpha_i = (a_{i1}, \dots , a_{ik}$) be  the normal vectors of hyperplanes $H_i$, $1 \leq i \leq n$, in the generic arrangement $\mathcal{A}$ in $\C^k$. Normal here is intended with respect to the usual dot product $$(a_1, \ldots, a_k)\cdot (v_1,\ldots, v_k)=\sum_i a_iv_i \quad .$$
Then the normal vectors to hyperplanes $D_L$, $L = \{s_1 < \dots < s_{k+1} \} \subset$ [$n$] in $ \mathbb{S} \simeq \mathbb{C}^n$ are nonzero vectors of the form 
\begin{equation}\label{eq:normvec}
\alpha_L = \sum^{k+1}_{i=1} (-1)^i \det (\alpha_{s_1}, \dots, \hat{\alpha_{s_i}}, \dots, \alpha_{s_{k+1}})e_{s_i} \quad ,
\end{equation} 
where $\{e_j\}_{1\leq j \leq n}$ is the standard basis of $\mathbb{C}^n$ (cf. \cite{BB}). 

Let $\mathcal{P}_{k+1}([n]) = \{L \subset [n] \mid |L| = k+1\}$ be the set of cardinality $k+1$ subsets of $[n]$. Following \cite{SoSuSi} we denote by
\begin{equation*}\label{An}
 A(\mathcal{A}_\infty) = (\alpha_L)_{L \in \mathcal{P}_{k+1} ([n])}
\end{equation*}
the matrix having in each row the entries of vectors $\alpha_L$ normal to hyperplanes $D_{L}$ and by $A_{\mathbb{T}}(\mathcal{A}_\infty)$ the submatrix  of $A(\mathcal{A}_\infty)$ with rows $\alpha_L$, $L \in \mathbb{T}$, $\mathbb{T} \subset \mathcal{P}_{k+1}([n])$. In this paper we are mainly interested in matrix $A_{\mathbb{T}}(\mathcal{A}_\infty)$ in the case of $\mathbb{T}$ good $6$-partition.

\subsection{Grassmannian $Gr(k,\CC^n)$}

Let $Gr(k, \CC^n)$ be the Grassmannian of $k$-dimensional subspaces of $\CC^n$ and 
\begin{eqnarray*}
\gamma: Gr(k, \CC^n) &\rightarrow& \mathbb{P}(\bigwedge^k \CC^n)  \\
<v_1, \dots, v_k> &\mapsto& [v_1 \wedge \dots \wedge v_k] \quad ,
\end{eqnarray*}
the Pl\"{u}cker embedding. Then $[x] \in \mathbb{P}(\bigwedge^k \CC^n)$ is in $\gamma(Gr(k, \CC^n))$ if and only if the map
\begin{eqnarray*}
\varphi_x : \CC^n &\rightarrow& \bigwedge^{k+1} \CC^n  \\ 
v &\mapsto& x \wedge v
\end{eqnarray*}
has kernel of dimension $k$, i.e. ker $\varphi_x = <v_1, \dots, v_k>$. If $e_1, \dots,e_n $ is a basis of $\CC^n$ then $e_I = e_{i_1}~\wedge~\dots~\wedge~e_{i_k}$, $I= \{ i_1, \dots, i_k \} \subset [n], i_1 < \dots < i_k$, is a basis for $\bigwedge^k \CC^n$ and $x \in \bigwedge^k \CC^n$ can be written uniquely as 

\begin{equation*}\label{betaformula}
x = \displaystyle \sum_{\substack {I \subseteq [n] \\ |I| = k}} \beta_I e_I 
= \sum_{1 \leq i_1 < \dots < i_k \leq n} \beta_{i_1 \dots i_k} (e_{i_1} \wedge \dots \wedge e_{i_k}) 
\end{equation*}
where homogeneous coordinates $\beta_I$ are the Pl\"{u}cker coordinates 
on $\mathbb{P}(\bigwedge^k \CC^n) \simeq \mathbb{P}^{\binom nk-1}(\CC)$ associated to 
the ordered basis $e_1, \dots, e_n$ of $\CC^n$. With this choice of basis for $\CC^n$ the matrix $M_x$ associated to $\varphi_x$ is a $\binom {n}{k+1} \times n$ matrix with rows indexed by subsets $I= \{ i_1, \dots, i_k \} \subset [n]$ and entries $b_{ij}=\left \{ \begin{array}{cc} (-1)^l \beta_{I\setminus\{j\}} & \mbox{ if } j=i_l \in I \\ 0 & \mbox{otherwise} \end{array} \right .$. Pl\"{u}cker relations, i.e conditions for dim(ker $\varphi_x$) = $k$, are vanishing conditions of all $(n-k+1) \times (n-k+1)$ minors of $M_x$. It is well known (see for instance \cite{harris}) that Pl\"{u}cker relations are degree 2 relations and they can also be written as 
\begin{equation}\label{pluck}
\sum_{l=0}^{k} (-1)^l \beta_{p_1 \dots p_{k-1}q_l} \beta_{q_0 \dots \hat{q_l} \dots q_k} =0
\end{equation}
for any $2k$-tuple $(p_1, \dots, p_{k-1}, q_0, \dots, q_k)$.

\begin{rem} Notice that vectors $\alpha_L$ in equation (\ref{eq:normvec}) normal to hyperplanes $D_L$ correspond to rows indexed by $L$ in the Pl\"{u}cker matrix $M_x$, that is $$A(\mathcal{A}_\infty)=M_x \quad ,$$ up to permutation of rows. 
Notice that, in particular, $\det (\alpha_{s_1}, \dots, \hat{\alpha_{s_i}}, \dots, \alpha_{s_{k+1}})$ is the Pl\"{u}cker coordinate $\beta_I$, $I = \{ s_1, s_2, \dots, s_{k+1} \} \backslash \{ s_i \}$.
\end{rem}


\subsection{Relation between intersections of lines in $\A_\infty$ and quadrics in $Gr(3, \CC^n)$}\label{relation}

Let $\A = \{H^0_1, \dots, H^0_n\}$ be a generic arrangement in $\CC^3$. If there exist $L_1, L_2, L_3 \subset [n]$ subsets of indices of cardinality $4$, such that codimension of $D_{L_1} \cap D_{L_2} \cap D_{L_3}$ is 2 then $\A$ is \textit{non very generic arrangement} (see \cite{BB}). \\
Let $\TT = \{L_1, L_2, L_3\}$ be a good 6-partition of indices $\{s_1, \dots, s_6\} \subset [n]$.  In \cite{sette}, authors proved that the codimension of $D_{L_1} \cap D_{L_2} \cap D_{L_3}$ is 2 if and only if points $\bigcap_{t \in L_1 \cap L_2} H_{\infty,t}$, $\bigcap_{t \in L_1 \cap L_3} H_{\infty,t}$ and $\bigcap_{t \in L_2 \cap L_3} H_{\infty,t}$ are collinear in $H_\infty$ (Lemma 3.1 \cite{sette}). 

Since $\alpha_{L_i}$ is vector normal to $D_{L_i}$, the codimension of $D_{L_1} \cap D_{L_2} \cap D_{L_3}$ is 2 if and only if rank $A_\TT(\A_\infty) = 2$, i.e. all $3 \times 3$ minors of $A_\TT(\A_\infty)$ vanish. In \cite{SoSuSi} authors proved the following Lemma.

\begin{lem}\label{lem:fin}(Lemma5.3 \cite{SoSuSi}) Let $\A$ be an arrangement of $n$ hyperplanes in $\CC^3$ and $\sigma.\mathbb{T} = \{ \{i_1, i_2, i_3, i_4\}$, $\{i_1, i_2, i_5, i_6\}, \{i_3, i_4, i_5, i_6\} \}$ a good $6$-partition of indices $s_1 < \ldots <s_6 \in [n]$ such that  $i_j=s_{\sigma(j)}$, $\sigma$ permutation in $\mathbf{S}_6$. Then $\rank A_{\sigma.\mathbb{T}}(\mathcal{A}_\infty)=2$ if and only if $\A$ is a point in the quadric of Grassmannian $Gr(3, \CC^n)$ of equation
\begin{equation}\label{eq:quadrica}
 \beta_{i_1 i_3 i_4} \beta_{i_2 i_5 i_6} - \beta_{i_2 i_3 i_4} \beta_{i_1 i_5 i_6} = 0  \quad .
\end{equation}
\end{lem}

\noindent
As consequence of above results, we obtain correspondence between points $x = \displaystyle \sum_{\substack {I \subset [n] \\ |I| = 3}}\beta_I e_I, \beta_I \neq 0,$ in quadric of equation (\ref{eq:quadrica}) and generic arrangements of $n$ hyperplanes $\A$ in $\CC^3$ such that $H_{\infty,i_1} \cap H_{\infty,i_2} $, $H_{\infty,i_3} \cap H_{\infty,i_4} $ and $H_{\infty,i_5} \cap H_{\infty,i_6} $ are collinear in $H_\infty$. Notice that condition $\beta_I \neq 0$ is direct consequence of $\A$ being generic arrangement.


\section{Motivating example of Pappus's Theorem for quadrics in $Gr(3, \CC^n)$}\label{motivation}

In classical projective geometry the following theorem is known as Pappus's theorem or Pappus's hhexagon theorem.

\begin{thm}[\textit{Pappus}]\label{pappus_1}
On a projective plane, consider two lines $l_1$ and $l_2$, and a couple of triple points $A, B, C$ and $A', B', C'$ which are on $l_1$ and $l_2$ respectively. Let $X, Y, Z$ be points of $AB' \cap A'B$, $AC' \cap A'C$ and $BC' \cap B'C$ respectively. Then there exists a line $l_3$ passing through the three points $X, Y, Z$ (see Figure \ref{fig1}).
\end{thm}

\begin{figure}[htbp]
 \begin{center}
  \includegraphics[width=50mm,bb =0 50 450 500]{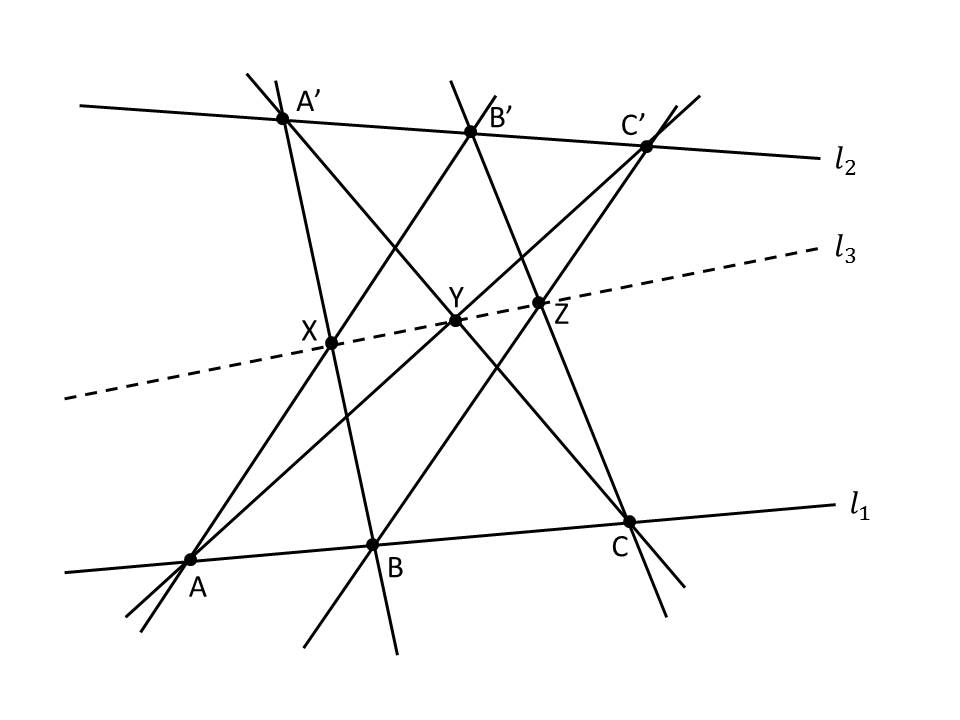}
 \end{center}
 \caption{Original Pappus's Theorem}
 \label{fig1}
\end{figure}

\noindent
This theorem was originally stated by Pappus of Alexandria around 290-350 A.D. .\\
In this section, we restate this classical theorem in terms of quadrics in Grassmannian. Indeed the six lines $AB', A'B$, $BC', B'C$, $AC', A'C$ $\in \PP^2(\CC)$ correspond to lines in the trace at infinity $\A_\infty$ of a generic arrangement $\A$ in $\CC^3$ and lines $l_1, l_2$ and $l_3$ correspond to collinearity conditions for intersection points of lines in $\A_\infty$.\\
\noindent
Consider a generic arrangement $\A = \{ H_1, \dots, H_6 \}$ of $6$ hyperplanes in $\CC^3$, $\mathcal{A}_\infty$ its trace at infinity and $\mathbb{T} = \{ L_1, L_2, L_3 \}$ the good $6$-partition defined by $L_1 = \{ 1, 2, 3, 4 \}$,  $L_2 = \{ 1, 2, 5, 6 \}$, $L_3 = \{ 3, 4, 5, 6 \}$. By Lemma\ref{lem:fin} we get that the triple points ${\displaystyle \bigcap_{i \in L_1 \cap L_2}} \bar H_i \cap H_\infty$, ${\displaystyle \bigcap_{i \in L_1 \cap L_3}} \bar H_i \cap H_\infty$, ${\displaystyle \bigcap_{i \in L_2 \cap L_3}} \bar H_i \cap H_\infty$ are collinear if and only if $\A$ is a point of the quadric
\begin{equation*}
{\rm Q}_1: \beta_{134}\beta_{256} - \beta_{234}\beta_{156} =0
\end{equation*}
in $Gr(3, \CC^6)$.\\
Analogously if $\mathbb{T}' = \{ L_1', L_2', L_3' \}$, $L_1' = \{ 4, 6, 2, 5 \}$, $L_2' = \{ 4, 6, 1, 3 \}$, $L_3' = \{ 2, 5, 1, 3 \}$ and $\mathbb{T}'' = \{ L_1'', L_2'', L_3'' \}$, $L_1'' = \{ 2, 4, 1, 6 \}$,  $L_2'' = \{ 2, 4, 3, 5 \}$, $L_3'' = \{ 1, 6, 3, 5 \}$ are different good 6-partitions then triple points ${\displaystyle \bigcap_{i \in L_1' \cap L_2'}} \bar H_i \cap H_\infty$, ${\displaystyle \bigcap_{i \in L_1' \cap L_3'}} \bar H_i \cap H_\infty$, ${\displaystyle \bigcap_{i \in L_2' \cap L_3'}} \bar H_i^{t_i} \cap H_\infty$ and ${\displaystyle \bigcap_{i \in L_1'' \cap L_2''}} \bar H_i \cap H_\infty$, ${\displaystyle \bigcap_{i \in L_1'' \cap L_3''}} \bar H_i \cap H_\infty$, ${\displaystyle \bigcap_{i \in L_2'' \cap L_3''}} \bar H_i \cap H_\infty$ are collinear if and only if $\A$ is, respectively, a point of quadrics
\begin{equation}
\begin{split}
&{\rm Q}_2: \beta_{425}\beta_{613} - \beta_{625}\beta_{413} =0 \mbox{ and }\\
&{\rm Q}_3: \beta_{216}\beta_{435} - \beta_{416}\beta_{235} =0  \quad .
\end{split}
\end{equation}
With above remarks and notations we can restate Pappus's theorem as follows (see Figure \ref{fig2}).
\begin{thm}(Pappus's theorem)
Let $\A = \{ H_1, \dots, H_6 \}$ be a generic arrangement of hyperplanes in $\CC^3$. If $\A$ is a point of two of three quadrics ${\rm Q}_1, {\rm Q}_2$ and ${\rm Q}_3$ in Grassmannian $Gr(3, \CC^6)$ , then $\A$ is also a point of the third.
In other words
$$
{\rm Q}_{i_1} \cap {\rm Q}_{i_2}  = \bigcap_{i=1}^3 {\rm Q}_i, \quad \{ i_1, i_2 \} \subset [3] . 
$$
\end{thm}

\begin{figure}[htbp] 
 \begin{center}
  \includegraphics[width=50mm,bb =0 50 450 500]{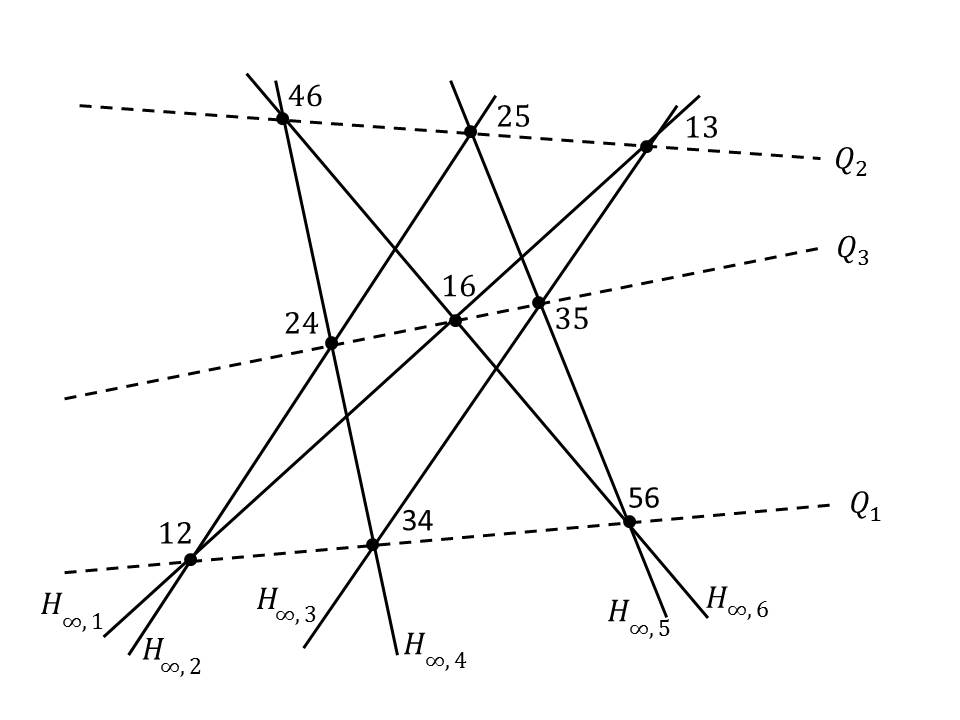}
 \end{center}
 \caption{Trace at infinity of $\mathcal{A} \in \bigcap_{i=1}^3 {\rm Q}_i$.
 {\footnotesize $ij$ denotes $H_{\infty,i} \cap H_{\infty,j}.$}}
\label{fig2}
\end{figure}

We develop this argument in the following sections.

\section{Pappus Variety} \label{pappusvariety}

In this section, we consider a generic arrangement $\{H_1, \dots, H_n \}$ in $\CC^3$ ($n \geq 6$). Let's introduce basic notations that we will use in the rest of the paper.

\begin{note} Let  $\{s_1, \dots, s_6\}$ be a subset of indices $\{1,\ldots,n\}$ and $\TT = \{ L_1, L_2, L_3 \}$ be the good $6$-partition given by $L_1 = \{s_1, s_2, s_3, s_4\}$, $L_2 = \{s_1, s_2, s_5, s_6\}$ and $L_3 = \{s_3, s_4, s_5, s_6\}$. Then for any permutation $\sigma \in {\bd S_6}$ we denote by $\sigma.\TT = \{ \sigma.L_1, \sigma.L_2, \sigma.L_3\}$ the good $6$-partition given by subsets $\sigma.L_1 = \{i_1, i_2, i_3, i_4\}, \sigma.L_2 = \{i_1, i_2, i_5, i_6\}, \sigma.L_3 = \{i_3, i_4, i_5, i_6\}$ with $i_j=s_{\sigma(j)}$. Accordingly, we denote by ${\rm Q}_{\sigma}$ the quadric in $Gr(3, \CC^n)$ of equation
\begin{center}
${\rm Q}_{\sigma} : \beta_{i_1i_3i_4}\beta_{i_2i_5i_6} -\beta_{i_2i_3i_4}\beta_{i_1i_5i_6} =0$ \quad.
\end{center}
\end{note}
\noindent
The following lemma holds.
\begin{lem}\label{lem:coincide}
Let $\sigma, \sigma' \in {\bd S_6}$ be distinct permutations, then ${\rm Q_{\sigma} = Q_{\sigma'}}$ if and only if there exists $\tau \in {\bd S_3}$ such that $\sigma.L_i \cap \sigma.L_j = \sigma'.L_{\tau(i)} \cap \sigma'.L_{\tau(j)}$ ($1\leq i<j\leq3$).
\end{lem}

\begin{proof}
By definition of good 6-partition we have that
\begin{eqnarray*}
L_1 &=& (L_1 \cap L_2) \cup (L_1 \cap L_3) \quad , \\
L_2 &=& (L_2 \cap L_1) \cup (L_2 \cap L_3) \quad , \\
L_3 &=& (L_3 \cap L_1) \cup (L_3 \cap L_2) \quad .
\end{eqnarray*}
Then there exists $\tau \in {\bd S_3}$ such that $\sigma$ and $\sigma'$ satisfy $\sigma.L_i \cap \sigma.L_j = \sigma'.L_{\tau(i)} \cap \sigma'.L_{\tau(j)}$ ($1\leq i<j\leq3$) if and only if $\sigma.L_l = \sigma'.L_{\tau(l)}$ for $l =1, 2, 3$, that is $A_{\sigma'.\TT}(\A_\infty)$ is obtained by permuting rows of $A_{\sigma.\TT}(\A_\infty)$. It follows that rank$A_{\sigma.\TT}(\A_\infty) = 2$ if and only if rank$A_{\sigma'.\TT}(\A_\infty) = 2$ and hence by Lemma \ref{lem:fin} this is equivalent to ${\rm Q}_\sigma \cap N_{s_1, \dots, s_6}$ $= {\rm Q}_{\sigma'} \cap N_{s_1, \dots, s_6}$, where $N_{s_1,\dots, s_6} = \{x = \displaystyle \sum_{\substack {I \subseteq [n] \\ |I| = 3}} \beta_I e_I \mid  \beta_I \neq 0 \mbox{ for any } I \subset \{s_1, \dots, s_6\}\}$. Since $N_{s_1, \dots, s_6}$ is dense open set in $\gamma(Gr(3, \CC^n))$, ${\rm Q}_\sigma \cap N_{s_1, \dots, s_6}$ $= {\rm Q}_{\sigma'} \cap N_{s_1, \dots, s_6}$ if and only if ${\rm Q}_\sigma$ $=  {\rm Q}_{\sigma'}$. Viceversa if ${\rm Q}_\sigma \cap N_{s_1, \dots, s_6}$ $= {\rm Q}_{\sigma'} \cap N_{s_1, \dots, s_6}$, then any generic arrangement $\A$ corresponding to a point in ${\rm Q}_\sigma \cap N_{s_1, \dots, s_6}$ corresponds to a point in ${\rm Q}_{\sigma'} \cap N_{s_1, \dots, s_6}$, that is rank$A_{\sigma.\TT}(\A_\infty) =2$ if and only if rank$A_{\sigma'.\TT}(\A_\infty) = 2$. It follows that  $A_{\sigma.\TT}(\A_\infty)$ and $A_{\sigma'.\TT}(\A_\infty)$ are submatrices of $A(\A_\infty)$ defined by the same three rows , i.e. $\sigma.L_l = \sigma'.L_{\tau(l)}$ for $l =1, 2, 3$. 
\end{proof}

\begin{defi}
For any 6 fixed indices $T=\{s_1,\cdots ,s_6\} \subset [n]$ the Pappus Variety is the hypersurface in $Gr(3, \CC^n)$ given by
$$
\mathcal{P}_T = \displaystyle \bigcup_{\sigma \in {\bd S_6}} {\rm Q}_{\sigma} \quad .
$$

\end{defi}
 
\bigskip

For $\sigma, \sigma' \in {\bd S_6}$ we define the equivalence relation $\sigma.\TT \sim \sigma'.\TT$ corresponding to ${\rm Q}_{\sigma} = {\rm Q}_{\sigma'}$ as following:
\begin{center}
$\sigma.\TT \sim \sigma'.\TT \Leftrightarrow \exists~\tau \in {\bd S_3}$ such that $\sigma.L_i \cap \sigma.L_j = \sigma'.L_{\tau(i)} \cap \sigma'.L_{\tau(j)}$($1\leq i<j\leq3$)  \quad.
\end{center}
We denote by $[\sigma]$ the equivalence class containing $\sigma.\TT$  and by ${\rm Q}_\sigma$ the corresponding quadric  (notice that $\sigma$ in notation ${\rm Q}_\sigma$ can be any representative of $[\sigma]$). By Lemma \ref{lem:coincide} $[\sigma]$ only depends on couples $L_i \cap L_j$ hence for each class $[\sigma]$ we can choice a representative $\tilde{\sigma}.\TT_0= \big\{ \{j_1, j_2, j_3, j_4\}$, $\{j_1, j_2, j_5, j_6\}$, $\{j_3, j_4, j_5, j_6\} \big\}$ such that  $j_1<j_2, j_3<j_4, j_5<j_6$ and $j_1<j_3<j_5$ and we can equivalently define 
$$
[\sigma] = \big\{ \{j_1, j_2\}, \{j_3, j_4\}, \{j_5, j_6\} \big\} \quad .
$$
Since the number of choices of $[\sigma]$ is $\dfrac {\binom{6}{2}\binom{4}{2}\binom{2}{2}}{3!} = 15$, Pappus variety is composed by 15 quadrics. Finally remark that $[\sigma] = \big\{ \{j_1, j_2\}, \{j_3, j_4\}, \{j_5, j_6\} \big\}$ and $[\sigma'] = \big\{ \{j_1', j_2'\}, \{j_3', j_4'\}, \{j_5', j_6'\} \big\}$ are ${\it disjoint}$, i.e. $[\sigma] \cap [\sigma'] = \emptyset$, if and only if $\{j_{2l-1}, j_{2l}\} \neq \{j'_{2l'-1}, j'_{2l'}\}$ for any $1 \leq l, l' \leq 3$.

\begin{defi}(Pappus configuration)
Let $[\sigma_1], [\sigma_2]$ and $[\sigma_3]$ be disjoint classes, a Pappus configuration is a set $\{ {\rm Q}_{\sigma_1}, {\rm Q}_{\sigma_2}, {\rm Q}_{\sigma_3} \}$ of quadrics in $Gr(3, \CC^n)$ such that
$$
{\rm Q}_{\sigma_{i_1}} \cap {\rm Q}_{\sigma_{i_2}} = \bigcap_{i=1}^3 {\rm Q}_{\sigma_i} \quad 
$$
for any $\{i_1, i_2\} \subset [3]$.
\end{defi}
\noindent
Quadrics ${\rm Q}_{\sigma_1}, {\rm Q}_{\sigma_2}, {\rm Q}_{\sigma_3}$ are said to be in Pappus configuration if $\{ {\rm Q}_{\sigma_1}, {\rm Q}_{\sigma_2}, {\rm Q}_{\sigma_3} \}$ is a Pappus configuration.

\begin{rem}\label{rem:disjointclass}
Fixed a class of good 6-partition  $[\sigma] = \big\{ \{j_1, j_2\}, \{j_3, j_4\}, \{j_5, j_6\} \big\}$, we shall count the number of disjoint classes. First let's count the number of classes $[\sigma'] = \big\{\{j_1', j_2'\}, \{j_3', j_4'\}, \{j_5', j_6'\} \big\}$ not disjoint and distinct from $[\sigma]$. Since $[\sigma]$ and $[\sigma']$ are distinct, only one couple $\{j_l', j_{l+1}' \}$ is contained in $[\sigma]$. Without lost of generality we can assume $\{j_l, j_{l+1}\} = \{j_1', j_2'\}$ ($l$ is either $1, 3$ or $5$) then pairs $\{j_3', j_4'\}$ and $\{j_5', j_6'\}$ are not in the same set, i.e. we have two possibilities:
\begin{center}
$\{j_3', j_5'\}$ and $\{j_4', j_6'\} \in [\sigma]$ \quad , \\
or \\
$\{j_3', j_6'\}$ and $\{j_4', j_5'\} \in [\sigma]$ \quad .
\end{center}
Hence there are $2 \cdot 3 + 1 = 7$ not disjoint classes from $[\sigma]$ and, since the number of all classes is $15$, we get that any fixed $[\sigma]$ admits exactly $15 - 7 = 8$ disjoint classes. 
\end{rem}
\section{Pappus's Theorem}\label{pappusthm}

In this section we restate Pappus's Theorem for quadrics in $Gr(3, \CC^n)$ by using notation introduced in previous section.
For a fixed class $[\sigma] = \big\{ \{j_1, j_2\}, \{j_3, j_4\}, \{j_5, j_6\} \big\}$ let's denote by $G_{[\sigma]}$ the free group generated by permutations of elements in each subset of $[\sigma]$, that is
$$
G_{[\sigma]} = \big\langle (j_{2l-1} \, j_{2l}) \in {\bd S_6} \, | \, l = 1, 2, 3 \, \big\rangle \quad ,
$$
and, for any class,  $[\sigma']$ let's define the set
$$
orbit_{G_{[\sigma]}}([\sigma']) = \big\{ \tau[\sigma'] \, | \, \tau \in G_{[\sigma]} \big\} \quad 
$$
where $\tau$ acts naturally as permutation of entries of each set in  $[\sigma']$.

\begin{rem} 
The action of $G_{[\sigma]}$ on class $[\sigma']$ disjoint from $[\sigma]$ is faithful. Indeed let $\tau, \tau' \in G_{[\sigma]}$ be such that $\tau[\sigma'] = \tau'[\sigma']$ then $\tau^{-1}\tau' [\sigma'] = [\sigma']$, i.e. $\tau^{-1}\tau' \in G_{[\sigma']}$. Thus we get $\tau^{-1}\tau' \in G_{[\sigma]} \cap G_{[\sigma']}$. Since $[\sigma]$ and $[\sigma']$ are disjoint, $G_{[\sigma]} \cap G_{[\sigma']} = \{ e \}$, i.e., $\tau = \tau'$. Remark that $|orbit_{G_{[\sigma]}}([\sigma'])| = |G_{[\sigma]}| = 8$ and $\tau[\sigma] = [\sigma]$ for any $\tau \in G_{[\sigma]}$.
\end{rem}

\begin{lem}\label{QT}
Let $[\sigma]$ and $[\sigma']$ be disjoint classes, then 
$$
orbit_{G_{[\sigma]}}([\sigma']) = \big\{ [\sigma''] \, | \,\, [\sigma] \cap [\sigma''] = \emptyset \big\} \quad .
$$
\end{lem}
\begin{proof}
First we prove that $orbit_{G_{[\sigma]}}([\sigma']) \subset \big\{ [\sigma''] \, | \,\, [\sigma] \cap [\sigma''] = \emptyset \big\}$. 
Let $[\sigma] = \big\{ \{j_1, j_2\}, \{j_3, j_4\}, \{j_5, j_6\} \big\}$ and $[\sigma'] = \big\{ \{j_1', j_2'\}, \{j_3', j_4'\}, \{j_5', j_6'\} \big\}$ be disjoint, then $| \{j_{2l-1}, j_{2l}\} \cap \{ j_{2m-1}', j_{2m}'\} |\leq 1$. Since $\tau \in G_{[\sigma]}$ permutes only $j_{2l-1}$ and $j_{2l}$ then $\tau[\sigma'] \cap [\sigma] = \emptyset$, that is $\tau[\sigma']$ is disjoint from $[\sigma]$, i.e. $\tau[\sigma'] \in \big\{ [\sigma''] \, | \,\, [\sigma] \cap [\sigma''] = \emptyset \big\}$.
Since $G_{[\sigma]}$ is faithful, $|orbit_{G_{[\sigma]}}([\sigma'])| = 8$ and, by calculations in Remark \ref{rem:disjointclass}, $|\big\{ [\sigma''] \, | \,\, [\sigma]\cap [\sigma''] = \emptyset \big\}| = 8$, it follows that $orbit_{G_{[\sigma]}}([\sigma'])$ = $\big\{ [\sigma''] \, | \,\, [\sigma] \cap [\sigma''] = \emptyset \big\}$.
\end{proof}

The following theorem holds.

\begin{thm}\label{thm:pappus}(Pappus's Theorem)
For any disjoint classes $[\sigma]$ and $[\sigma']$, there exists a unique class $[\sigma'']$ disjoint from $[\sigma]$ and $[\sigma']$ such that $\{ {\rm Q}_{\sigma}$, ${\rm Q}_{\sigma'}, {\rm Q}_{\sigma''} \}$ is a Pappus configuration.
\end{thm}

\begin{proof}
Following example in section \ref{motivation}, for any class $[\omega_1] = \big\{\{j_1, j_2\},$ $\{j_3, j_4\}, \{j_5, j_6\} \big\}$ let's consider disjoint classes $[\omega_2] = \big\{ \{ j_1, j_3\}, \{j_2, j_5\}, \{j_4, j_6\} \big\}$ and  $[\omega_3] = \big\{ \{ j_1, j_6\}, \{j_2, j_4\}, \{j_3, j_5\} \big\}$.

The corresponding quadrics have equations:
\begin{equation*}
\begin{split}
&{\rm Q}_{\omega_1} : \beta_{j_1j_3j_4}\beta_{j_2j_5j_6} - \beta_{j_2j_3j_4}\beta_{j_1j_5j_6} = 0 \quad ,\\
&{\rm Q}_{\omega_2} : \beta_{j_4j_2j_5}\beta_{j_6j_1j_3} - \beta_{j_6j_2j_5}\beta_{j_4j_1j_3} = 0 \quad ,\\
&{\rm Q}_{\omega_3} : \beta_{j_5j_1j_6}\beta_{j_3j_2j_4} - \beta_{j_3j_1j_6}\beta_{j_5j_2j_4} = 0 \quad .
\end{split}
\end{equation*}
By definition of $\beta_{ijk}$, equations of ${\rm Q}_{\omega_2}$ and ${\rm Q}_{\omega_3}$ can equivalently be written as 
$$
{\rm Q}_{\omega_2} : \beta_{j_2j_4j_5}\beta_{j_1j_3j_6} + \beta_{j_2j_5j_6}\beta_{j_1j_3j_4} = 0 \quad ,
$$
$$
{\rm Q}_{\omega_3} : \beta_{j_1j_5j_6}\beta_{j_2j_3j_4} + \beta_{j_1j_3j_6}\beta_{j_2j_4j_5} = 0 \quad .
$$
If we denote left side of defining equations of ${\rm Q}_{\omega_i}$ by $P_{\omega_i}$ then
\begin{equation*}\label{eq:quadrics}
P_{\omega_2} - P_{\omega_1} = P_{\omega_3} \quad ,
\end{equation*}
that is zeros of any two polynomials $P_{\omega_{i_1}}, P_{\omega_{i_2}}$ are zeros of $P_{\omega_{i_3}}$, $\{i_1, i_2, i_3\} = \{1, 2, 3\}$. We get
$$
{\rm Q}_{\omega_{i_1}} \cap {\rm Q}_{\omega_{i_2}} = \bigcap_{i=1}^3 {\rm Q}_{\omega_i} 
$$
for any $\{i_1, i_2\} \subset [3]$, i.e. ${\rm Q}_{\omega_1}, {\rm Q}_{\omega_2}$ and ${\rm Q}_{\omega_3}$ are in Pappus configuration.\\
By Lemma \ref{QT}, since $[\omega_1] \cap [\omega_2] = \emptyset$, the set of disjoint classes from $[\omega_1]$ is given by
$$
 \big\{ [\sigma_0] \, | \,\, [\omega_1] \cap [\sigma_0]  = \emptyset \big\}=\big\{ \tau_0[\omega_2] \,| \, \, \tau_0 \in G_{[\omega_1]} \big\}  \quad .
$$
Then if $[\sigma']$ is disjoint from $[\omega_1]$, there exists a unique element $\tau \in G_{[\omega_1]}$ such that $[\sigma'] = \tau[\omega_2]$. That is, for a generic class $[\omega_1]$, any disjoint couple $([\omega_1], [\sigma'])$ is of the form $([\omega_1], \tau[\omega_2]) = (\tau[\omega_1], \tau[\omega_2])$ and we have

$$
{\rm Q}_{\omega_1} = {\rm Q}_{\tau \omega_1} : \beta_{\tau(j_1)\tau(j_3)\tau(j_4)}\beta_{\tau(j_2)\tau(j_5)\tau(j_6)} - \beta_{\tau(j_2)\tau(j_3)\tau(j_4)}\beta_{\tau(j_1)\tau(j_5)\tau(j_6)} = 0 \quad ,
$$
$$
{\rm Q}_{\sigma'} = {\rm Q}_{\tau \omega_2} : \beta_{\tau(j_4)\tau(j_2)\tau(j_5)}\beta_{\tau(j_6)\tau(j_1)\tau(j_3)} - \beta_{\tau(j_6)\tau(j_2)\tau(j_5)}\beta_{\tau(j_4)\tau(j_1)\tau(j_3)} = 0\quad .
$$
By antisymmetric property of indices of $\beta_{ijk}$, if we denote by $P_{\omega_1}$ and $P_{\sigma'}$ the left side of above equations, i.e. 
\begin{equation*}
\begin{split}
&P_{\omega_1} = \beta_{\tau(j_1)\tau(j_3)\tau(j_4)}\beta_{\tau(j_2)\tau(j_5)\tau(j_6)} - \beta_{\tau(j_2)\tau(j_3)\tau(j_4)}\beta_{\tau(j_1)\tau(j_5)\tau(j_6)} \quad ,\\
& P_{\sigma'} = \beta_{\tau(j_4)\tau(j_2)\tau(j_5)}\beta_{\tau(j_6)\tau(j_1)\tau(j_3)} - \beta_{\tau(j_6)\tau(j_2)\tau(j_5)}\beta_{\tau(j_4)\tau(j_1)\tau(j_3)} 
\end{split}
\end{equation*}
then 
$$
P_{\sigma''} := P_{\sigma'} - P_{\omega_1}  = \beta_{\tau(j_5)\tau(j_1)\tau(j_6)}\beta_{\tau(j_3)\tau(j_2)\tau(j_4)} - \beta_{\tau(j_3)\tau(j_1)\tau(j_6)}\beta_{\tau(j_5)\tau(j_2)\tau(j_4)} \quad 
$$
is the defining polynomial of ${\rm Q}_{\tau \omega_3}$. That is $[\sigma'']$ is uniquely determined by disjoint couple $([\omega_1], [\sigma'])$. 
\end{proof}
\noindent
From proof of Theorem \ref{thm:pappus} we get that for any class $[\omega_1] = \big\{\{j_1, j_2\},$ $\{j_3, j_4\}, \{j_5, j_6\} \big\}$ if we denote $[\omega_2] = \big\{ \{ j_1, j_3\}, \{j_2, j_5\}, \{j_4, j_6\} \big\}$ and  $[\omega_3] = \big\{ \{ j_1, j_6\}, \{j_2, j_4\}, \{j_3, j_5\} \big\}$, then all Pappus configurations are of the form $\{ {\rm Q}_{\tau \omega_1}, {\rm Q}_{\tau \omega_2}, {\rm Q}_{\tau \omega_3}\}$, $\tau \in G_{[\omega_1]}$ and the following Corollary holds.\\

\begin{cor}
The number of  Pappus configurations $\{ {\rm Q}_{\sigma}, {\rm Q}_{\sigma'}, {\rm Q}_{\sigma''}\}$ in $Gr(3, \CC^6)$ is $20$.
\end{cor}

\begin{proof}
By Remark \ref{rem:disjointclass} the number of $[\sigma]$ is $15$ and by Lemma \ref{QT} each fixed class $[\sigma]$ admits $8$ disjoint classes. By Theorem \ref{thm:pappus} if $[\sigma]$ and $[\sigma']$ are fixed, $[\sigma'']$ is uniquely determined, thus the number of the sets $\big\{ [\sigma], [\sigma'], [\sigma''] \big\}$ is $15 \times 8 / 3! = 20$.
\end{proof}

\section{Intersections of quadrics}\label{intersection}

In this section we study intersections of quadrics in $Gr(3, \CC^n)$. In particular we are interested in intersections of sets
$$
{\rm Q}_{\sigma}^\circ = {\rm Q}_\sigma \cap \big \{ x = \displaystyle \sum_{\substack {I \subset [n] \\ |I| = 3}} \beta_I e_I  | \, \beta_I \neq 0, \mbox{ for any } I \subset \{s_1, \dots, s_6\} \big\} 
$$
of points in quadrics ${\rm Q}_\sigma$ that correspond to arrangements  of lines in $\PP^2(\CC)$ with subarrangement $\{H_{s_1}, \ldots, H_{s_6}\}$ generic. 
The following lemma holds.

\begin{lem}\label{prop:onlyzero}
If $[\sigma_1], [\sigma_2], [\sigma_3]$ are distinct and pairwise not disjoint classes then ${\rm Q}_{\sigma_1}^\circ \cap {\rm Q}_{\sigma_2}^\circ \cap {\rm Q}_{\sigma_3}^\circ = \emptyset$.
\end{lem}

\begin{proof}
If $[\sigma_1], [\sigma_2], [\sigma_3]$ are not disjoint then either
\begin{enumerate}
\renewcommand{\labelenumi}{(\arabic{enumi})}
\item $|[\sigma_1] \cap [\sigma_2] \cap [\sigma_3]| = 1$\label{enu:nonempty} \quad or
\item $|[\sigma_{i_1}] \cap [\sigma_{i_2}] |= 1 \,  (1\leq i_1 < i_2 \leq 3)$ and $[\sigma_1] \cap [\sigma_2] \cap [\sigma_3] = \emptyset$\label{enu:empty} \quad .
\end{enumerate}

\bigskip
(\ref{enu:nonempty}) Assume $[\sigma_1] \cap [\sigma_2] \cap [\sigma_3] = \{i_1, i_2\}$. Let $[\sigma_1] = \big\{ \{i_1,i_2\}, \{i_3,i_4\}, \{i_5,i_6\} \big\}, [\sigma_2] = \big\{ \{i_1,i_2\}, \{i_3,i_5\}, \{i_4,i_6\} \}$, and $[\sigma_3] = \big\{ \{i_1,i_2\}, \{i_3,i_6\}, \{i_4,i_5\} \big\}$ then we obtain the following quadrics 
\begin{center}
${\rm Q}_{\sigma_1} : \beta_{i_1i_3i_4}\beta_{i_2i_5i_6} - \beta_{i_2i_3i_4}\beta_{i_1i_5i_6} = 0$ \quad , \\
${\rm Q}_{\sigma_2} : \beta_{i_1i_3i_5}\beta_{i_2i_4i_6} - \beta_{i_2i_3i_5}\beta_{i_1i_4i_6} = 0$ \quad , \\
${\rm Q}_{\sigma_3} : \beta_{i_1i_3i_6}\beta_{i_2i_4i_5} - \beta_{i_2i_3i_6}\beta_{i_1i_4i_5} = 0$ \quad .
\end{center}
Any point $x \in {\rm Q}_{\sigma_1}^\circ \cap {\rm Q}_{\sigma_2}^\circ$ belongs to $Gr(3, \CC^n)$, that is $x$ satisfies Pl\"{u}cker relations in (\ref{pluck}). In particular $x \in Pl_1 \cap Pl_2$ where $Pl_1$ and $Pl_2$ are the quadrics:
\begin{center}
$Pl_1 : \beta_{i_1i_3i_2}\beta_{i_4i_5i_6} - \beta_{i_1i_3i_4}\beta_{i_2i_5i_6} + \beta_{i_1i_3i_5}\beta_{i_2i_4i_6} - \beta_{i_1i_3i_6}\beta_{i_2i_4i_5} = 0$ \quad , \\
$Pl_2 : \beta_{i_2i_3i_1}\beta_{i_4i_5i_6} - \beta_{i_2i_3i_4}\beta_{i_1i_5i_6} + \beta_{i_2i_3i_5}\beta_{i_1i_4i_6} - \beta_{i_2i_3i_6}\beta_{i_1i_4i_5} = 0$ \quad .
\end{center}
Notice that $Pl_1$ and $Pl_2$ can be obtained from equations in (\ref{pluck}) considering the 6-tuples $(p_1, p_2, q_0, q_1, q_2, q_3) = (i_1, i_3, i_2, i_4, i_5, i_6)$ and $(i_2, i_3, i_1, i_4, i_5, i_6)$ respectively. We get 

\begin{center}
${\rm Q}_{\sigma_2} - {\rm Q}_{\sigma_1} - Pl_1 + Pl_2 : \beta_{i_1i_3i_6}\beta_{i_2i_4i_5} - \beta_{i_2i_3i_6}\beta_{i_1i_4i_5} + 2(\beta_{i_1i_2i_3}\beta_{i_4i_5i_6}) = 0$ \quad .
\end{center} 
Since $\beta_{i_1i_2i_3} \neq 0$ and $\beta_{i_4i_5i_6} \neq 0$ then $\beta_{i_1i_2i_3}\beta_{i_4i_5i_6} \neq 0$
and hence $\beta_{i_1i_3i_6}\beta_{i_2i_4i_5} - \beta_{i_2i_3i_6}\beta_{i_1i_4i_5} \neq 0$, that is $x  \not\in {\rm Q}_{\sigma_3}^\circ$.

\bigskip
(\ref{enu:empty}) Assume $[\sigma_1] \cap [\sigma_2] = \{i_1,i_2\}$, $[\sigma_1] \cap [\sigma_3] = \{i_3,i_4\}$ and $[\sigma_2] \cap [\sigma_3] = \{i_5,i_6\}$ and name $P_1 = \{i_1,i_2\}$, $P_2 = \{i_3,i_4\}$, $P_3 = \{i_5,i_6\}$. To any point $x \in {\rm Q}_{\sigma_1}^\circ \cap {\rm Q}_{\sigma_2}^\circ \cap {\rm Q}_{\sigma_3}^\circ$ corresponds the existence of an arrangement with a generic sub-arrangement indexed by $\{i_1, \ldots , i_6\}$ which trace at infinity $\{H_{\infty,i_1}, \ldots , H_{\infty,i_6}\}$ satisfyies collinearity conditions as in Figure  \ref{fig3}. That is there exist couples $P_4 \in [\sigma_1], P_5 \in [\sigma_2]$ and $P_6 \in [\sigma_3]$ that correspond, respectively, to intersection points $p_4,p_5$ and $p_6$ of lines in $\{H_{\infty,i_1}, \ldots , H_{\infty,i_6}\}$ (see Figure  \ref{fig3}).

\begin{figure}[htbp]
 \begin{center}
  \includegraphics[width=70mm,bb =0 0 600 500]{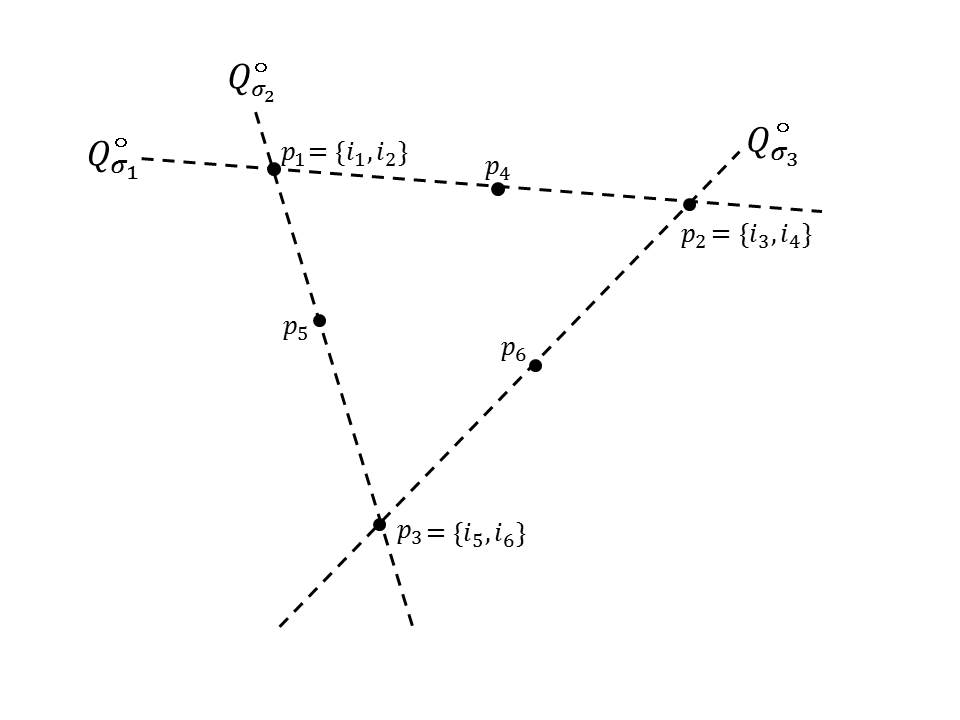}
 \end{center}
 \caption{Case (2) trace at infinity of $\A \in \bigcap_{i=1}^3 {\rm Q}^\circ_{\sigma_i}$, $\{i,j\}$ corresponds to $H_{\infty,i} \cap H_{\infty,j}$.}
 \label{fig3}
\end{figure}
\noindent
By definition of $P_1$, $P_2$ and $P_3$ we have
$$
P_3 = \{i_5, i_6\} \in \big(\{i_1, \dots, i_6\} \backslash P_1\big) \cap \big(\{i_1, \dots, i_6\} \backslash P_2\big) \quad .
$$
On the other hand, if $P_4$ is different from $P_1$ and $P_2$ in ${\rm Q}^\circ_{\sigma_1}$ then  $P_4=\big(\{i_1, \dots, i_6\} \backslash P_1\big)$ $\cap \big(\{i_1, \dots, i_6\} \backslash P_2\big)$. Thus we get $P_3 = P_4$ and, similarly, $P_5 = P_2$ and $P_6 = P_1$, that is ${\rm Q}^\circ_{\sigma_1} = {\rm Q}^\circ_{\sigma_2} = {\rm Q}^\circ_{\sigma_3}$ which contradict hypothesis.
\end{proof}

\begin{lem}\label{lem:alldisjoint}
For any three pairwise disjoint classes $[\sigma_1], [\sigma_2], [\sigma_3]$, either $\{ {\rm Q}_{\sigma_1}, {\rm Q}_{\sigma_2}, {\rm Q}_{\sigma_3} \}$ is a Pappus configuration or $\displaystyle\bigcap_{i=1}^3 {\rm Q}_{\sigma_i}^\circ = \emptyset$.
\end{lem}

\begin{proof}
By Pappus's Theorem, for any two disjoint classes $[\sigma_i], [\sigma_j]$, there exists $[\sigma_{ij}]$ such that $\{ {\rm Q}_{\sigma_i}, {\rm Q}_{\sigma_j}, {\rm Q}_{\sigma_{ij}} \}$ is Pappus configuration. If $[\sigma_{ij}] = [\sigma_k]$ for some $k \in [3]$, then $\{ {\rm Q}_{\sigma_1}, {\rm Q}_{\sigma_2}, {\rm Q}_{\sigma_3} \}$ is a Pappus configuration. Thus assume all $[\sigma_{ij}] \neq [\sigma_k]$ for any $k = 1, 2, 3$. Moreover $[\sigma_{12}], [\sigma_{13}], [\sigma_{23}]$ are distinct since if $[\sigma_{ij}] = [\sigma_{ik}]$ then $[\sigma_j] = [\sigma_k]$.\\
If $[\sigma_{12}] \cap [\sigma_{13}] \neq \emptyset$, $[\sigma_{12}] \cap [\sigma_{23}] \neq \emptyset$ and $[\sigma_{13}] \cap [\sigma_{23}] \neq \emptyset$, then $\bigcap_{1 \leq l_1 < l_2 \leq 3} {\rm Q}_{\sigma_{l_1l_2}}^\circ = \emptyset$ by Lemma \ref{prop:onlyzero} and $\displaystyle\bigcap_{i=1}^3 {\rm Q}_{\sigma_i}^\circ = (\displaystyle\bigcap_{i=1}^3 {\rm Q}_{\sigma_i}^\circ) \cap (\displaystyle\bigcap_{1 \leq l_1 < l_2 \leq 3} {\rm Q}_{\sigma_{l_1l_2}}^\circ) = \emptyset$.\\
Otherwise assume $[\sigma_{12}] \cap [\sigma_{13}] = \emptyset$, we get a new Pappus configuration. Since the number of disjoint classes is finite, iterating the process, we will eventually get 3 classes $[\sigma_{l_1}], [\sigma_{l_2}]$, $[\sigma_{l_3}]$ pairwise not disjoint and  $\displaystyle\bigcap_{i=1}^3 {\rm Q}_{\sigma_i}^\circ = (\displaystyle\bigcap_{i=1}^3 {\rm Q}_{\sigma_i}^\circ) \cap {\rm Q}_{\sigma_{l_1}}^\circ \cap {\rm Q}_{\sigma_{l_2}}^\circ \cap  {\rm Q}_{\sigma_{l_3}}^\circ = \emptyset$.
\end{proof}

\begin{lem}
If $[\sigma_1], [\sigma_2], [\sigma_3]$ are distinct classes such that $[\sigma_1] \cap [\sigma_2] \neq \emptyset$ and $[\sigma_i] \cap [\sigma_3] = \emptyset$ for $i = 1, 2$, then $\displaystyle\bigcap_{i=1}^3 {\rm Q}_{\sigma_i}^\circ = \emptyset$.
\end{lem}

\begin{proof}
Since $[\sigma_1], [\sigma_3]$ and $[\sigma_2], [\sigma_3]$ are disjoint, there exist $[\sigma_4]$ and $[\sigma_5]$ such that $\{{\rm Q}_{\sigma_1}, {\rm Q}_{\sigma_3}, {\rm Q}_{\sigma_4}\}$ and $\{{\rm Q}_{\sigma_2}, {\rm Q}_{\sigma_3}, {\rm Q}_{\sigma_5}\}$ are Pappus configurations and
\begin{center}
$[\sigma_1] \cap [\sigma_5] \neq \emptyset,$ $[\sigma_2] \cap [\sigma_4] \neq \emptyset,$
$[\sigma_4] \cap [\sigma_5] \neq \emptyset$ \quad .
\end{center}
Indeed if one of them is empty, we obtain $3$ disjoin classes not in Pappus configuration and by Lemma \ref{lem:alldisjoint}, it follows $\displaystyle\bigcap_{i=1}^3 {\rm Q}_{\sigma_i}^\circ = \displaystyle\bigcap_{i=1}^5 {\rm Q}_{\sigma_i}^\circ = \emptyset$.
Since $[\sigma_1] \cap [\sigma_2] \neq \emptyset$, we can assume $\{ i_1, i_2 \} = [\sigma_1] \cap [\sigma_2]$ and we can set 
\begin{center}
$[\sigma_1] = \big\{ \{i_1, i_2\}, \{i_3, i_4\}, \{i_5, i_6\} \big\}$, $[\sigma_2] = \big\{ \{i_1, i_2\}, \{i'_3, i'_4\}, \{i'_5, i'_6\} \big\}$, $[\sigma_3] = \big\{ \{j_1, j_2\}, \{j_3, j_4\}, \{j_5, j_6\} \big\}$ \quad .
\end{center}
To any point $x \in \displaystyle\bigcap_{i=1}^3 {\rm Q}_{\sigma_i}^\circ \neq \emptyset$ corresponds an arrangement $\A$ with generic subarrangement $\{H_{i_1}, \ldots , H_{i_6}\}$ with trace at infinity $\{H_{\infty,i_1}, \ldots , H_{\infty,i_6}\}$ intersecting as in Figures \ref{fig:pappus_sigma1} and \ref{fig:pappus_sigma2} ( up to rename).
\begin{figure}[htbp]
 \begin{center}
  \includegraphics[width=70mm,bb =0 0 600 500]{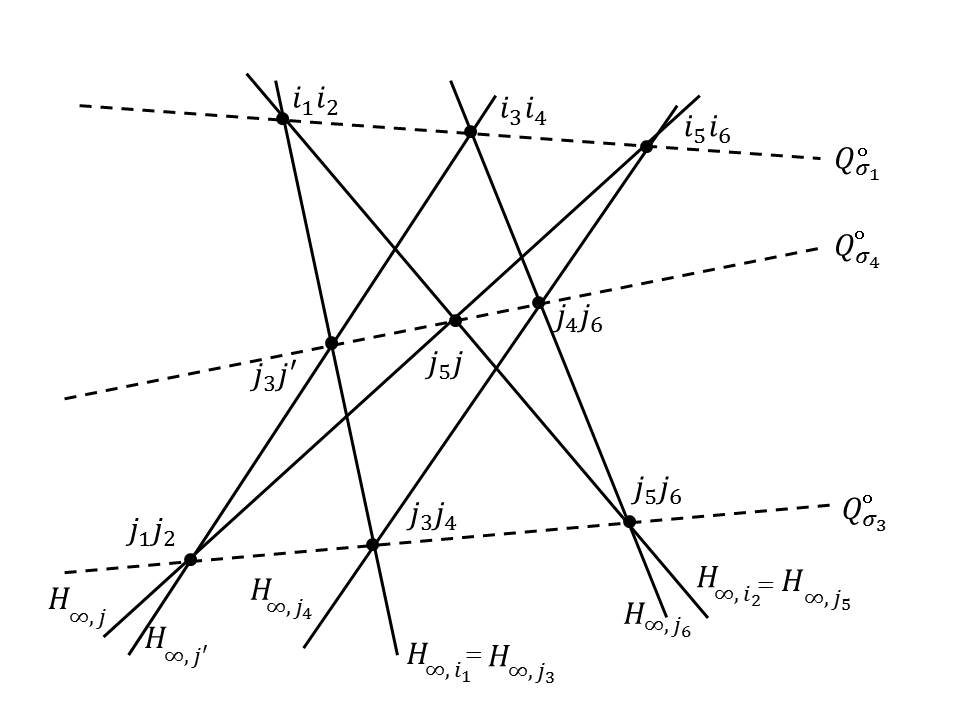}
 \end{center}
 \caption{each $j, j'$ is $j_1$ or $j_2$. }
 \label{fig:pappus_sigma1}
\end{figure}
It follows that $\{j_4, j_6\} \in [\sigma_4]$ and since $\{j_3, j_5\} = \{i_1, i_2\} \in [\sigma_1]$ and $[\sigma_1] \cap [\sigma_4] = \emptyset$ (see Figure \ref{fig:pappus_sigma1}), there are two possibilities:
\begin{center}
$[\sigma_4] = \big\{ \{j_4, j_6\}, \{j_1, j_3\}, \{j_2, j_5\} \big\}$ \\
or \\
$[\sigma_4] = \big\{ \{j_4, j_6\}, \{j_1, j_5\}, \{j_2, j_3\} \big\}$.
\end{center}
Analogously (see Figure \ref{fig:pappus_sigma2}) class $[\sigma_5]$ is of the form
\begin{center}
$[\sigma_5] = \big\{ \{j_4, j_6\}, \{j_1, j_3\}, \{j_2, j_5\} \big\}$ \\
or \\
$[\sigma_5] = \big\{ \{j_4, j_6\}, \{j_1, j_5\}, \{j_2, j_3\} \big\}$.
\end{center}
\begin{figure}[htbp]
 \begin{center}
  \includegraphics[width=70mm,bb =0 0 600 500]{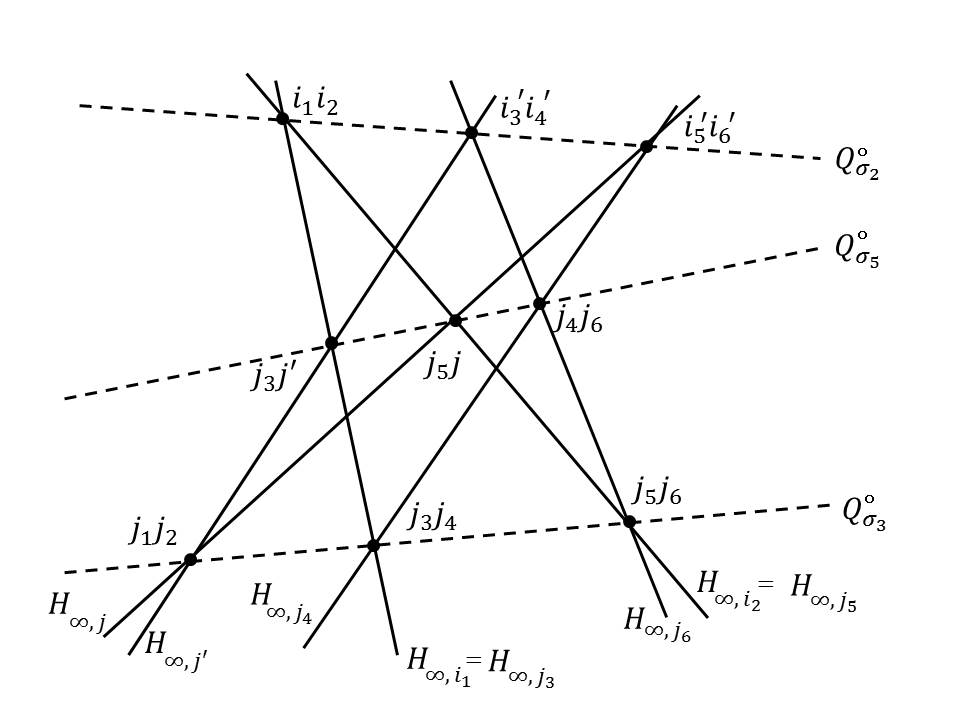}
 \end{center}
 \caption{each $j, j'$ is $j_1$ or $j_2$. }
 \label{fig:pappus_sigma2}
\end{figure}
Since $[\sigma_1] \cap  [\sigma_5] \neq \emptyset$ and $[\sigma_5] \not\ni \{j_3, j_5\} = \{i_1, i_2\}$, we deduce that $\{j_4, j_6\} = \{i_3, i_4\}$ or $\{i_5, i_6\}$, which is not possible by $[\sigma_1] \cap [\sigma_4] = \emptyset$. Hence $\displaystyle\bigcap_{i=1}^3 {\rm Q}_{\sigma_i}^\circ = \emptyset$.
\end{proof}
\noindent
Notice that the Hesse arrangement in $\PP^2(\CC)$ (see Figure \ref{fig4}) can be regarded as a generic arrangement of $6$ lines which intersection points satisfy $6$ collinearity conditions.
\begin{figure}[htbp]
 \begin{center}
  \includegraphics[width=70mm,bb =0 0 600 500]{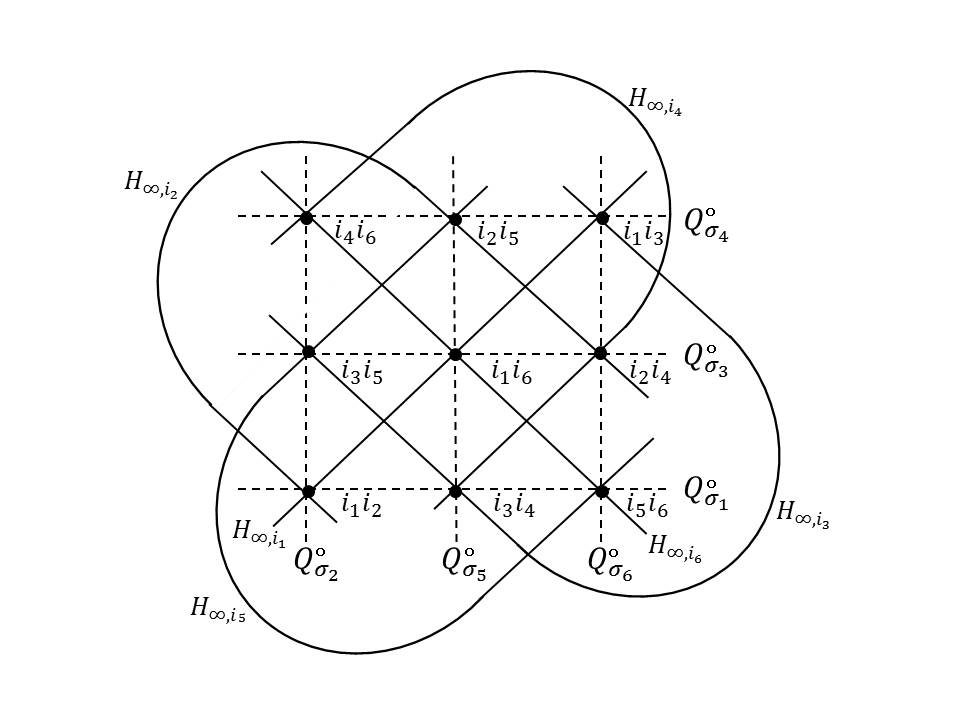}
 \end{center}
 \caption{Hesse arrangement with $H_{\infty i_1}, \dots H_{\infty i_6}$ and $\bigcap_{i=1}^6 {\rm Q}^\circ_{\sigma_i} \neq \emptyset$.}
 \label{fig4}
\end{figure}
\begin{defi}(Hesse configuration)
Let $[\sigma_i], 1 \leq i \leq 6$ be distinct classes, we call Hesse configuration a set $\{ {\rm Q}_{\sigma_1}, \dots, {\rm Q}_{\sigma_6} \}$ of quadrics in $Gr(3, \CC^n)$ such that there exist disjoint sets $I, J \subset [6]$, $|I| = |J| = 3$ such that $\{ {\rm Q}_{\sigma_i} \}_{i \in I}, \{ {\rm Q}_{\sigma_j}\}_{j \in J}$ are Pappus configurations and $[\sigma_i] \cap [\sigma_j] \neq \emptyset$ for any $i \in I$, $j \in J$.
\end{defi}
With above notations, the following classification Theorem holds.

\begin{thm}\label{thm:conjecture}
For any choice of indices $\{s_1,\ldots, s_6\} \subset [n]$ sets ${\rm Q}_{\sigma}^\circ,~\sigma \in \bd S_6$,  in Grassmannian $Gr(3, \CC^n)$
intersect as follows.
\begin{enumerate}
\renewcommand{\labelenumi}{(\arabic{enumi})}
\item For any disjoint classes $[\sigma_1]$ and $[\sigma_2]$, there exist $[\sigma_3], \dots, [\sigma_6]$ such that $\{{\rm Q}_{\sigma_1}, \ldots ,{\rm Q}_{\sigma_6}\}$ is an Hesse configuration for $I = \{1, 2, 3\}, J = \{4, 5, 6\}$ and 
\begin{equation*}
\displaystyle\bigcap_{i=1}^2 {\rm Q}_{\sigma_i}^\circ = \displaystyle\bigcap_{i=1}^3 {\rm Q}_{\sigma_i}^\circ \supsetneq \displaystyle\bigcap_{i=1}^4 {\rm Q}_{\sigma_i}^\circ \supsetneq \displaystyle\bigcap_{i=1}^6 {\rm Q}_{\sigma_i}^\circ \supsetneq \emptyset \quad .
\end{equation*} 
\item For any not disjoint classes $[\sigma_1]$ and $[\sigma_2]$, there exist $[\sigma_3], \dots, [\sigma_6]$ such that $\{{\rm Q}_{\sigma_1}, \ldots ,{\rm Q}_{\sigma_6}\}$ is an Hesse configuration for $I = \{1, 3, 4\}, J = \{2, 5, 6\}$ and 
\begin{equation*}
\displaystyle\bigcap_{i=1}^2 {\rm Q}_{\sigma_i}^\circ \supsetneq \displaystyle\bigcap_{i=1}^3 {\rm Q}_{\sigma_i}^\circ = \displaystyle\bigcap_{i=1}^4 {\rm Q}_{\sigma_i}^\circ \supsetneq \displaystyle\bigcap_{i=1}^6 {\rm Q}_{\sigma_i}^\circ \supsetneq \emptyset \quad .
\end{equation*}
\end{enumerate}

All other intersections are empty.

\end{thm}

\begin{rem}
Notice that, since Hesse configuration only exists in complex case, in $Gr(3, \CC^n)$ we can find $6$ quadrics $\{{\rm Q}_{\sigma_1}, \dots, {\rm Q}_{\sigma_6}\}$ such that
$$
\displaystyle\bigcap_{i=1}^6 {\rm Q}_{\sigma_i}^\circ \supsetneq \emptyset \quad ,
$$
while in $Gr(3, \RR^n)$, 
$$
\displaystyle\bigcap_{j\in J \subset [6], \, |J|>4} {\rm Q}_{\sigma_j}^\circ = \emptyset \quad .
$$
It follows that in real case, for any choice of indices $\{s_1, \ldots, s_6\} \subset [n]$, we have at most 4 collinearity conditions (see Figure \ref{pap_config2}) corresponding to $15$ hyperplanes in Discriminantal arrangement with 4 multiplicity 3 intersections in codimension 2 (see Figure \ref{codim2_2}). While in complex case Hesse configuration (see Figure \ref{fig4}) gives rise to a Discriminantal arrangement containing $15$ hyperplanes intersecting in 6 multiplicity 3 spaces in codimension 2. \\
This remark allows a better understanding of differences in combinatorics of Discriminantal arrangement in real and complex case. Moreover those observations suggest that some special configuration of lines in projective plane intersecting in a big number of triple points could be understood by studying Discriminantal arrangements with maximum number of multiplicity 3 intersections in codimension 2.
\begin{figure}[htbp]
  \includegraphics[width=40mm,bb =0 50 450 500]{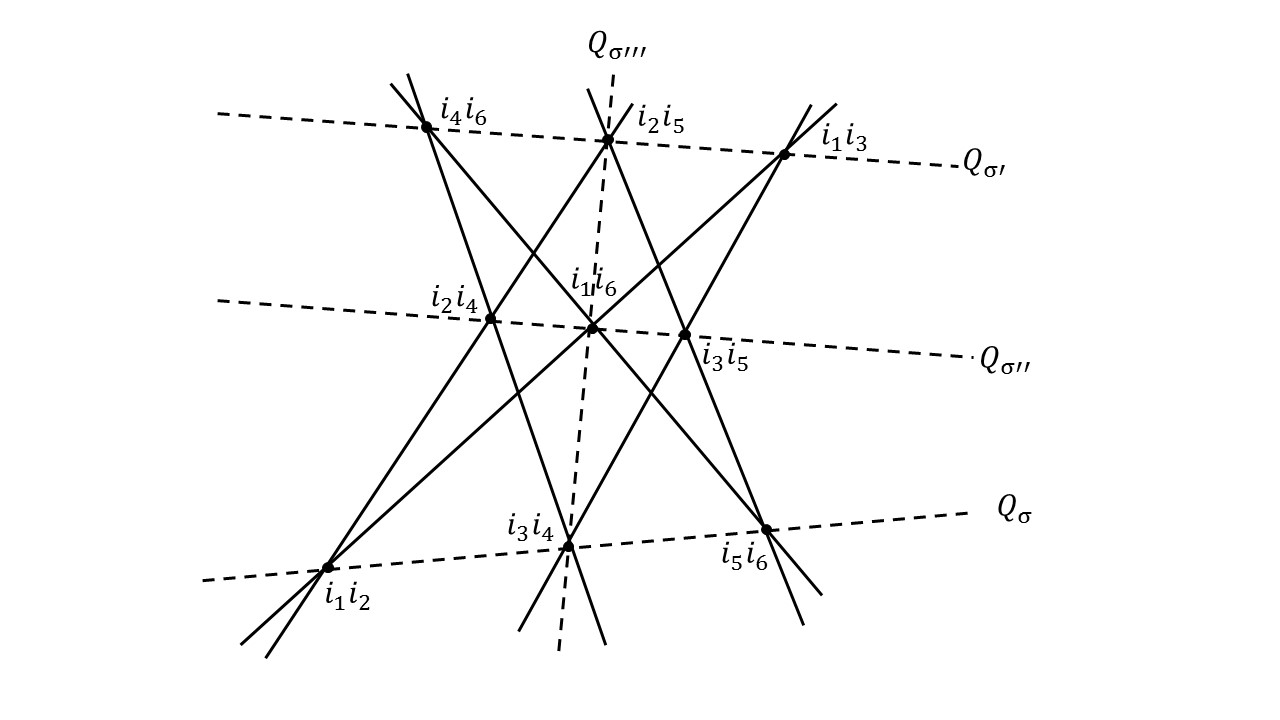}
 \caption{Generic arrangement $\A$ in $\RR^3$ containing $6$ lines satisfying $4$ collinearity conditions.}
 \label{pap_config2}
\end{figure}
\begin{figure}[htbp]
 \begin{center}
  \includegraphics[width=40mm,bb =0 50 450 500]{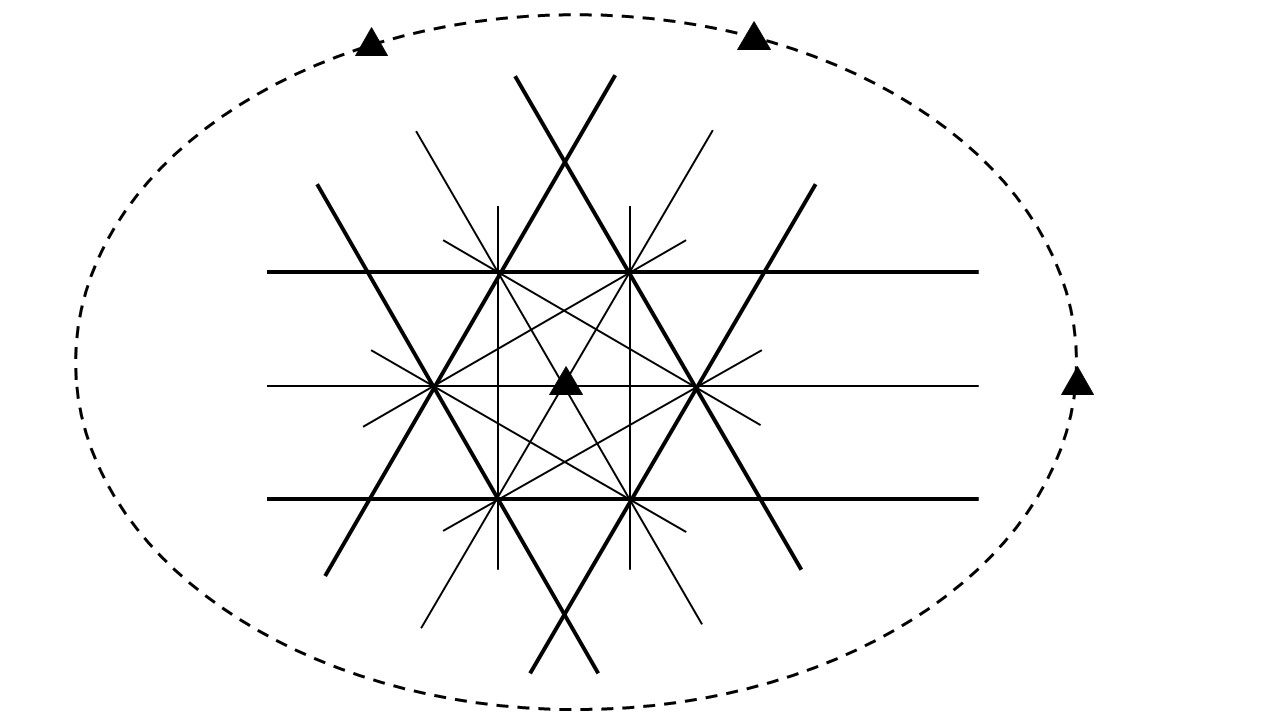}
 \end{center}
 \caption{Codimension $2$ intersections of $15$ hyperplanes in $\B(n,3, \A_\infty)$ indexed in $\{s_1,\ldots, s_6\} \subset [n]$ with $4$ multiplicity $3$ points $\blacktriangle$ corresponding to intersections $\bigcap_{i=1}^{3} D_{\sigma.{L_i}}$, $\bigcap_{i=1}^{3} D_{\sigma'.{L_i}}$, $\bigcap_{i=1}^{3} D_{\sigma''.{L_i}}$ and $\bigcap_{i=1}^{3} D_{\sigma'''.{L_i}}$, $\sigma, \sigma', \sigma'', \sigma'''$ as in Figure \ref{pap_config2}.}
 \label{codim2_2}
\end{figure}
\end{rem}


\end{document}